\documentclass[12pt]{amsart}

\usepackage{amsfonts,amsthm,amsmath,amssymb,latexsym}
\usepackage{graphicx}
\usepackage{eucal}
\usepackage[all]{xy}

\title{On complex extension of the Liouville map}

\address{XD: Department of Mathematics, Kingsborough CC of the City University of New York, 2001 Oriental Blvd, Brooklyn, NY 11235}

\email{Xinlong.Dong@kbcc.cuny.edu}

\author{Xinlong Dong and Dragomir \v Sari\' c}
\address{DS: Department of Mathematics, Queens College of the City University of New York, 65--30 Kissena Blvd., Flushing, NY 11367}
\address{DS: Mathematics PhD. Program, Graduate Center of the City University of New York, 365 Fifth Avenue, New York, NY 10016-4309}
\email{Dragomir.Saric@qc.cuny.edu}

\thanks{The second author was partially supported by the collaboration grant  346391 from the Simons Foundation and a PSC-CUNY research grant.}

\newtheorem{thm}{Theorem}

\newtheorem{lem}[thm]{Lemma}

\theoremstyle{definition}

\theoremstyle{plain}

\newcommand{\D}{\mathbb D}
\newcommand{\C}{\mathbb C}

\newcommand{\R}{\mathbb R}
\newcommand{\HH}{\mathbb H}

\newcommand{\T}{\mathcal  T}

\newcommand{\HHH}{\mathcal H}
\newcommand{\W}{\mathbf  W}
\newcommand{\LLL}{\mathcal  L}

\renewcommand{\leq}{\leqslant}
\renewcommand{\geq}{\geqslant}
\renewcommand{\epsilon}{\varepsilon}
\renewcommand{\phi}{\varphi}

\subjclass{}

\keywords{}
\date{\today}

\begin{document}

\begin{abstract}
The Liouville map assigns to each point in the Teichm\"uller space a positive Radon measure on the space of geodesics of the universal covering of the base Riemann surface. This construction which was introduced by Bonahon is valid for both finite and infinite Riemann surfaces. Bonahon and S\"ozen proved that the Liouville map is differentiable for closed Riemann surfaces and the second author extended this result to all other Riemann surfaces. Otal proved that the Liouville map is real analytic using an idea from the geometric analysis. The purpose of this note is to give another proof of Otal's result using a complex analysis approach.
\end{abstract}

\maketitle

A Riemann surface $X$ is {\it conformally hyperbolic} if its universal covering is the upper half-plane $\HH =\{ z=x+iy:y>0\}$. The hyperbolic metric on $X$ is the projection of the hyperbolic metric $\frac{|dz|}{y}$
on $\HH$ by the covering map. The Teichm\"uller space $\T (X)$ is the space of quasiconformal maps from the base surface $X$ onto variable Riemann surfaces up to isometries and homotopies. When $X$ is a closed or finite area conformally hyperbolic Riemann surface, the Teichm\"uller space $\T (X)$ is a finite-dimensional complex manifold and when $X$ is not of finite area then $\T (X)$ is an infinite-dimensional complex Banach manifold.

The space of geodesics $G(\tilde{X})$ of the universal covering $\tilde{X}$ of a conformally hyperbolic Riemann surface $X$ 
supports a natural {\it Liouville measure} which is the unique (up to scalar multiple) measure of full support that is invariant under the isometries of $\tilde{X}$. A $\pi_1(X)$-invariant Radon measure on $G(\tilde{X})$ is called a {\it geodesic current} for $X$.   
Bonahon \cite{Bonahon} introduced the {\it Liouville map} $$\LLL :\T (X)\to \mathcal{G}(X)$$ from the Teichm\"uller space $\T (X)$ to the space of geodesic currents $\mathcal{G} (X)$ that assigns to each 
quasiconformal deformation $[f:X\to Y] \in \T (X)$ of the Riemann surface X the pullback of the Liouville measure of $Y$ under the deformation $f$.

When the Riemann surface $X$ is compact, Bonahon \cite{Bonahon} used the Liouville map to introduce an alternative description of the Thurston boundary to $\T (X)$. Under the assumption that the Riemann surface X is compact, Bonahon and S\"ozen \cite{BonahonSozen} considered a notion of differentiability for the Liouville map $\LLL: \T (X) \to \mathcal{G} (X)$, by extending $\mathcal{G} (X)$ to a topological vector space $\mathcal{H} (X)$ of $\pi_1(S)$-invariant H\"older distributions on the space of geodesics $G(\tilde{X})$; see \S \ref{sec:Holder} for precise definitions. In this framework, they proved that $\LLL$ is continuously differentiable.

The second author \cite{Saric}, and more recently, Bonahon and the second author \cite{BonahonSaric} introduced the Thurston boundary to Teichm\"uller spaces of arbitrary conformally hyperbolic Riemann surfaces. In order to account for the control over deformations of an arbitrary conformally hyperbolic Riemann surface $X$, the second author \cite{Saric} has introduced a family of semi-norms on the space of geodesic currents $\mathcal{G}({X})$ as follows. 

For a fixed $0<\lambda\leq 1$, a {\it test function} $\xi :G(\tilde{X})\to\mathbb{R}$ is a real-valued $\lambda$-H\"older continuous function with compact support on $G(\tilde{X})$ such that its  support can be mapped by a (normalizing) isometry $\gamma$ of $\tilde{X}$ to a fixed compact subset $K$ of $G(\tilde{X})$ and $\|\xi\circ \gamma^{-1}\|_{\lambda}\leq 1$, where the $\lambda$-norm of $\xi$ is given by $\|\xi\|_{\lambda}:=\sup_{g\in G(\tilde{X})} |\xi (g) |+\sup_{g_1\neq g_2}\frac{|\xi (g_1)-\varphi (g_2)|}{d(g_1,g_2)^{\lambda}}$. The distance $d(g_1,g_2)$ is the maximum of the distances between the pairs of endpoints of $g_1$ and $g_2$ on the ideal boundary $\partial\tilde{X}$ of $\tilde{X}$ in the angle metric with respect to a fixed point in $X$; see \S \ref{sec:Holder} for more details.

The family of semi-norms $\{\|\cdot\|_{\lambda}\}_{0<\lambda\leq 1}$ on the space of geodesic currents $\mathcal{G}(X)$ for $X$ is defined by 
$$
\|\alpha\|_{\lambda}=\sup_{\xi}\Big{|}\int_{G(\tilde{X})}\xi  d\alpha\Big{|}
$$
where $\alpha\in\mathcal{G}(X)$ and the supremum is over all test functions $\xi$ for the exponent $\lambda$
 (see \cite{Saric}, \cite{Saric1} for details). The space of {\it bounded geodesic currents} $\mathcal{G}_b(X)$ for $X$ consists of all $\alpha\in\mathcal{G} (X)$ such that, for every $0<\lambda\leq 1$, 
 $$
 \|\alpha\|_{\lambda}<\infty
 .$$

The tangent vectors to $\T (X)$ are identified with the Zygmund bounded functions on the ideal boundary $\partial\tilde{X}$ of the universal covering $\tilde{X}$ (for example, see \cite[\S 16]{GardinerLakic}). A Zygmund function induces a finitely additive signed measure on $G(\tilde{X})$ which, in general, is not countably additive. Therefore to discuss differentiability of the Liouville map, a larger topological vector space that contains $\mathcal{G}_b(X)$ is needed.
The second author \cite{Saric1} introduced a family of spaces $\mathcal{H}^{\lambda}(X)$ of $\pi_1(X)$-invariant  H\"older distributions on $G(\tilde{X})$ for  $0<\lambda\leq 1$. Each $W\in\mathcal{H}^{\lambda}(X)$ is a $\pi_1(X)$-invariant linear functional on the space of $\lambda$-H\"older continuous functions with compact support that satisfies
$$
\| W\|_{\lambda}:=\sup_{\xi} |W(\xi )|<\infty 
$$
where the supremum is over all test functions for the exponent $\lambda$.

The second author \cite{Saric1}
 proved that the Liouville map $\LLL :\T (X)\to\mathcal{H}^{\lambda}(X)$
 is continuously differentiable in this general setting for each $0<\lambda\leq 1$. Otal \cite{Otal} proved that the Liouville map is real analytic using the above construction of the spaces $\mathcal{H}^{\lambda}(X)$.

\begin{thm}[Otal \cite{Otal}]
The Liouville map
$$
\LLL :\T (X)\to\mathcal{H}^{\lambda} (X)
$$
is real analytic for each $0<\lambda\leq 1$.
\end{thm}

The description of the topology on $\mathcal{G}_b(X)$ using the $\lambda$-norms is somewhat complicated due to the fact that the test functions need to have $\lambda$-norm less than $1$ under the pre-compositions of the normalizing isometries and their supports need to be of certain size. In addition, one is then required to take the supremum over all test functions for a fixed H\"older exponent $\lambda$.

More recently, Bonahon and the second author \cite{BonahonSaric} introduced a simpler family of semi-norms on $\mathcal{G}_b(X)$ that describes the same topology on the space of bounded geodesic currents $\mathcal{G}_b(X)$ for an arbitrary conformally hyperbolic Riemann surface $X$. Given a continuous function $\zeta :G(\tilde{X})\to\mathbb{R}$ with compact support, we define a semi-norm $\|\cdot\|_{\zeta}$ on $\mathcal{G}_b(X)$ as follows. For $\alpha\in \mathcal{G}_b(X)$ set
$$
\|\alpha\|_{\zeta}=\sup_{\gamma}\Big{|}\int_{G(\tilde{X})} \zeta \circ\gamma d\alpha\Big{|}
$$
where the supremum is over all isometries $\gamma$ of $\tilde{X}$. Using the topology on $\mathcal{G}_b(X)$ induced by the family of semi-norms $\{\|\cdot\|_{\zeta}\}$, Bonahon and the second author \cite{BonahonSaric} proved that the Liuoville map $\LLL :\T (X)\to\mathcal{G}_b(X)$ is an embedding onto its image and that the boundary of the projectivization of $\LLL (\T (X))$ is exactly equal to the space of projective bounded measured laminations - a Thurston boundary to $\T (X)$.

With this in mind, we introduce a new  space of H\"older distributions on $G(\tilde{X})$ and give an alternative proof to the above theorem of Otal for this new target space. Given a H\"older continuous function $\xi :G(\tilde{X})\to\mathbb{C}$ with a compact support, define
$$
\|\alpha\|_{\xi}=\sup_{\gamma}\Big{|}\int_{G(\tilde{X})}\xi\circ\gamma d\alpha \Big{|}
$$
for any $\alpha\in\mathcal{G}_b(X)$, where the supremum is over all isometries $\gamma$ of $\tilde{X}$. 

Let $H(\tilde{X})$ be the space of all H\"older continuous functions $\xi :G(\tilde{X})\to\mathbb{C}$ with compact support. 
The space of {\it bounded H\"older distributions} $\mathcal{H}_b(X)$ for $X$ consists of all linear functionals $W:H(\tilde{X})\to\mathbb{C}$ such that,  for all $A\in\pi_1(X)$ and $\xi\in H(\tilde{X})$, $$W (\xi\circ A)=W(\xi )$$ and, for each $\xi\in H(\tilde{X})$,
$$
\| W\|_{\xi}:=\sup_{\gamma}|W(\xi\circ\gamma )|<\infty
$$
where the supremum is over all isometries $\gamma$ of $\tilde{X}$. 

The space of bounded H\"older distributions $\mathcal{H}_b(X)$ is a complex topological vector space whose topology is induced by the finite intersections of balls for the semi-norms $\{\| \cdot \|_{\xi}\}_{\xi}$ (see \cite[\S 2.3]{BonahonSaric}). The notions of complex and real analytic maps in this paper are in terms of this topological vector space structure on $\mathcal{H}_b(X)$. 
We prove

\begin{thm}
\label{thm:main}
The Liouville map
$$
\LLL :\T (X)\to \mathcal{H}_b(X)
$$
is real analytic.
\end{thm}

Otal \cite{Otal} extended the Liouville map from the Teichm\"uller space $\T (X)$ to its open neighborhood in the quasi-Fuchsian space $\mathcal{QF}(X)$ and proved that the extension is holomorphic in an appropriate sense. This implies that the Liouville map is real analytic. The existence of the extension of the Liouville map to a neighborhood of the Teichm\"uller space $\T (X)$ is proved using geometric analysis ideas. Namely, the extended Liouville map acts as a linear functional on the twice differentiable functions and the extension to the H\"older continuous functions is established by showing that the norm of the extension to twice differentiable functions is bounded in terms of their H\"older norms. We replace this argument by finding the rate of decay of the imaginary part of the cross-ratio of four points under quasiconformal maps which are not too far from fixing the real line (see Lemma \ref{lem:liouville-complex} and \ref{lem:infinitesimal-liouville}). This may be of independent interest since we obtain a direct estimate of the cross-ratio of four points whereas traditionally the estimates were given for one point at a time. 

\vskip .2 cm

{\it Acknowledgements.} The second author is grateful to Francis Bonahon for the many conversations over the years regarding geodesic currents. Both authors are indebted to the referee for careful reading and useful comments.

\section{The Teichm\"uller and quasi-Fuchsian spaces}

In this section we introduce the Teichm\"uller and Quasi-Fuchsian space, their manifold structures and a  realization of the Teichm\"uller space as a totally real submanifold of the Quasi-Fuchsian space.

Any conformally hyperbolic Riemann surface $X$ is identified  with $\HH /\Gamma$, where $\Gamma <PSL_2(\R )$ is a Fuchsian group acting on the upper half-plane $\HH$. Consider the family of all quasiconformal map $f:\HH \to\HH$ that conjugate the Fuchsian group $\Gamma$ onto another Fuchsian group and fix $0$, $1$ and $\infty$. Two quasiconformal maps $f,f_1:\HH\to\HH$ from the family are {\it Teichm\"uller equivalent} if they agree on the ideal boundary $\hat{\R}=\R\cup\{\infty\}$ of the upper half-plane $\HH$. The Teichm\"uller space $\T(X)$ consists of all Teichm\"uller equivalence classes $[f]$. 

The Beltrami coefficient of a quasiconformal map $f:\HH\to\HH$ is given by $\mu =\frac{f_{\bar{z}}}{ f_z}$ and it satisfies $\|\mu\|_{\infty}<1$. Conversely, given $\mu\in L^{\infty}(\HH )$ with $\|\mu\|_{\infty}<1$ there exists a unique (normalized) quasiconformal map $f:\HH \to\HH$ that fixes $0$, $1$ and $\infty$ whose Beltrami coefficient is $\mu$. A quasiconformal map $f$ conjugates $\Gamma$ onto another Fuchsian group if and only if $\mu\circ\gamma\frac{\overline{\gamma'}}{\gamma'}=\mu$ for all $\gamma\in \Gamma$. 
Two Beltrami coefficients are {\it Teichm\"uller equivalent} if the corresponding normalized quasiconformal maps are equal on $\hat{\R}$.  Therefore we can define $\T (X)$ to be a set of all Teichm\"uller classes $[\mu]$ of Beltrami coefficients that satisfy $\mu\circ\gamma (z)\frac{\overline{\gamma'(z)}}{\gamma'(z)}=\mu (z)$ for all $\gamma\in \Gamma$ and $z\in\HH$ (for example, see \cite{GardinerLakic}).

A quasiconformal map $f:\hat{\mathbb{C}}\to\hat{\mathbb{C}}$ that conjugates  $\Gamma<PSL_2(\mathbb{R})$ onto a subgroup of $PSL_2(\mathbb{C})$ fixing $0$, $1$ and $\infty$ represents an element of the quasi-Fuchsian space $\mathcal{QF}(\Gamma )$. Two quasiconformal maps $f$ and $g$ fixing $0$, $1$ and $\infty$ that conjugate $\Gamma$ onto a subgroup of $PSL_2(\C )$ are {\it equivalent} if they agree on $\hat{\R}=\R\cup\{\infty\}$. Denote by $[f]\in\mathcal{QF}(\Gamma )$ the corresponding equivalence class. Equivalently, we can define $\mathcal{QF}(\Gamma )$ to consist of all equivalence classes $[\mu ]$ of Beltrami coefficients on $\C$ where two Beltrami coefficients are equivalent if their corresponding normalized quasiconformal maps agree on $\hat{\R }$. 
If $\Gamma$ is trivial then $f$ is a quasiconformal map fixing $0$, $1$ and $\infty$ which does not necessarily preserve the upper half-plane $\HH$. 

A quasiconformal map $f:\HH\to\HH$ extends by the reflection $f(z)=\overline{f(\bar{z})}$ for $z$ in the lower half-plane $\HH^{-}=\{ z:Im(z)<0\}$ to a quasiconformal map of $\hat{\C}$. The corresponding Beltrami coefficient satisfies $\mu (z)=\overline{\mu (\bar{z})}$ for $z$ in the lower half-plane $\HH^{-}$. The Teichm\"uller space $\T (X)$ embeds into the quasi-Fuchsian space $\mathcal{QF}(X)$ by extending each Beltrami coefficient $\mu$ on $\HH$ to $\C$ using the reflection in the real line.

Bers introduced a complex Banach manifold structure to the Teichm\"uller space $\T (X)$. The complex chart around the basepoint $[0 ]\in \T (X)$ is obtained as follows. Let $\tilde{\mu}$ be the Beltrami coefficient which equals $\mu$ in the upper half-plane $\HH$ and equals zero in the lower half-plane $\HH^{-}$. The solution $f=f^{\tilde{\mu}}$ to the Beltrami equation $ f_{\bar{z}}=\tilde{\mu}  f_z$  is conformal in the lower half-plane $\HH^{-}$. The Schwarzian derivative 
$$
S(f^{\tilde{\mu}})(z)=\frac{(f^{\tilde{\mu}})'''(z)}{(f^{\tilde{\mu}})'(z)}-\frac{3}{2}\Big{(}\frac{(f^{\tilde{\mu}})''(z)}{(f^{\tilde{\mu}})'(z)}\Big{)}^2
$$
for $z\in\HH^{-}$ defines a holomorphic function $\phi (z)=S(f^{\tilde{\mu}})(z)$ which satisfies
$(\phi \circ\gamma )(z) \gamma'(z)^2=\phi (z)$ and $\|\varphi\|_{b}:=\sup_{z\in\HH^{-}}|y^2\phi (z)|<\infty$, called a {\it cusped form} for $X$. The space of all cusped forms $\phi :\HH^{-}\to\C$ for $X$ is a complex Banach space $\mathcal{Q}_b(X)$ with the norm $\|\cdot\|_b$ (see \cite{GardinerLakic}).  

The Schwarzian derivative maps the unit ball in $L^{\infty}(\HH )$ onto an open subset of $\mathcal{Q}_b(X)$ and it projects to a homeomorphism $\Phi$ from $\T (X)$ to an open subset of $\mathcal{Q}_b(X)$ containing the origin. 
The open ball $B_{[0]}(\frac{1}{2}\log 2)$ in $\T (X)$ of radius $\frac{1}{2}\log 2$ and center $[0]$ maps under $\Phi$ onto an open set in $\mathcal{Q}_b(X)$ which contains the ball of radius $\frac{2}{3}$ and is contained in the ball of radius $2$ with center $0\in\mathcal{Q}_b(X)$. The map $\Phi :B_{[0]}(\frac{1}{2}\log 2)\to \mathcal{Q}_b(X)$ is a chart map for the base point $[0]\in\T (X)$ (see \cite[\S 6]{GardinerLakic}). The Ahlfors-Weill section provides an explicit formula for $\Phi^{-1}$ on the ball of radius $\frac{1}{2}$ and center $0$ in $\mathcal{Q}_b(X)$. Namely if $\phi\in \mathcal{Q}_b(X)$ with $\|\phi\|_{b}<\frac{1}{2}$ then Ahlfors and Weill prove that $\Phi^{-1}(\phi )=[-2y^2\phi (\bar{z})]$ (see \cite[\S 6]{GardinerLakic}). The Beltrami coefficient $\eta_{\phi}(z):=-2y^2\phi (\bar{z})$  is said to be {\it harmonic}. The Ahlfors-Weill formula gives an explicit expression of Beltrami coefficients that are representing points in $\T (X)$ corresponding  to the  holomorphic disks $\{t\phi :|t|<1,\|\phi\|_{b}<\frac{1}{2}\}$ in the chart in $\mathcal{Q}_b(X)$, namely $$\Phi^{-1}(\{t\phi :|t|<1\})=\{ [t\eta_\phi ]\in\T (X) :|t|<1\}.$$

Let $[\mu_0]\in \T (X)$ be fixed and define $X^{[\mu_0]}=\mathbb{H}/ f^{\mu_0}\Gamma (f^{\mu_0})^{-1}$ to be the image Riemann surface. A chart for $[\mu_0]\in\T (X)$ is given by the chart of $[0]\in\T (X^{[\mu_0]})$ under the {\it translation map} $T_{[\mu_0 ]}:\T (X)\to \T (X^{[\mu_0]})$ defined by
$$
T_{[\mu_0 ]}([f])=[f\circ (f^{\mu_0})^{-1}]
$$
for $[f]\in\T (X)$. Namely, let $B^{[\mu_0]}_{[0]}(\frac{1}{2}\log 2)\subset \T (X^{[\mu_0]})$ be a ball with center $[0]$ and radius $\frac{1}{2}\log 2$ in the Teichm\"uller space of $X^{[\mu_0]}$. Let $\Phi_{[\mu_0]}$ be the Bers embedding for $\T (X^{[\mu_0]})$. Then the chart for $[\mu_0]$ is given by
$$
\Phi_{[\mu_0]}\circ T_{[\mu_0]}:T_{[\mu_0]}^{-1}(B^{[\mu_0]}_{[0]}(\frac{1}{2}\log 2))\to \mathcal{Q}_b(X^{[\mu_0]})
$$
where $T_{[\mu_0]}^{-1}(B^{[\mu_0]}_{[0]}(\frac{1}{2}\log 2))=B_{[\mu_0]}(\frac{1}{2}\log 2)$ because the translation map $T_{[\mu_0]}$ is an isometry for the Teichm\"uller metric. If $\{ t\phi:|t|<1\}\subset \mathcal{Q}_b(X^{[\mu_0]})$ using the Ahlfors-Weill section and the chain rule we have
$$
(\Phi_{[\mu_0]}\circ T_{[\mu_0]})^{-1}(t\phi )=\Big{[}\frac{ t(\eta_\phi\circ f^{\mu_0}(z))
\frac{\overline{f^{\mu_0}_{{z}}(z)}}{f^{\mu_0}_z(z)}+{\mu_0 (z)}}{1+ t(\eta_\phi\circ f^{\mu_0}(z))\frac{\overline{f^{\mu_0}_{{z}}(z)}}{f^{\mu_0}_z(z)}\overline{\mu_0(z)}}\Big{]}.
$$

By the Bers simultaneous uniformization theorem, the quasi-Fuchsian space $\mathcal{QF}(X)$ is identified with the product of the Teichm\"uller space $\T (X)$ of the Riemann surface $X$ and the Teichm\"uller space $\T (\bar{X})$ of the mirror image Riemann surface $\bar{X}$. A point $[\mu ]\in \mathcal{QF}(X)$ is represented by a Beltrami coefficient $\mu$ on $\C$ and we identify it with the pair of (equivalence classes of) Beltrami coefficients $([\mu_1],[\mu_2])$, where $\mu_1=\mu |_{\HH}$ and $\mu_2=\mu |_{\HH^{-}}$.  The Teichm\"uller space $\T (\bar{X})$ is defined as the equivalence classes of Beltrami coefficients in $\HH^{-}$ invariant under $\Gamma$ where two Beltrami coefficients are equivalent if their corresponding normalized quasiconformal maps from $\HH^{-}$ onto itself agree on $\hat{\R}$. The charts on $\mathcal{QF}(X)$ are given by the product  charts on $\T (X)\times\T (\bar{X})$ and the inverse of the charts are given by the product of the Ahlfors-Weil sections. The Teichm\"uller space $\T(X )$ embeds into $\mathcal{QF}(X)$ as a totally real-analytic submanifold by the formula $[\mu ]\mapsto [\tilde{\mu}]$, where $\tilde{\mu}(z)=\mu (z)$ for $z\in\HH$ and $\tilde{\mu }(z)=\overline{\mu (\bar{z})}$ for $z\in\HH^{-}$.

Given a pair of Beltrami coefficients $(\mu_1 ,\mu_2 )$ with $\mu_1\in L^{\infty}(\HH )$ and $\mu_2\in L^{\infty}(\HH^{-})$ define $\tilde{\mu}_1=\mu_1$ on $\HH$, $\tilde{\mu}_1=0$ on $\HH^{-}$, and $\tilde{\mu}_2=\mu_2$ on $\HH^{-}$, $\tilde{\mu}_2=0$ on $\HH$. The Schwarzian derivatives $S(f^{\tilde{\mu}_1})$ and  $S(f^{\tilde{\mu}_2})$ give a pair of holomorphic functions $(\phi_1,\phi_2)$ in $\HH^{-}$ and $\HH$, respectively. The holomorphic functions $\phi_1$ and $\phi_2$ satisfy the invariance property under $\Gamma$ and the boundedness in the $\|\cdot\|_b
$-norm in their respective domains. Let $\eta_i$ be the harmonic Beltrami coefficient corresponding to $\phi_i$ under the Ahlfors-Weill section. Then the pair of harmonic Beltrami coefficients $\eta:=(\eta_1,\eta_2)$ represents the point in $\mathcal{QF}(X)$ that is mapped to $\phi :=(\phi_1,\phi_2)$ under the Bers embedding. Such $\eta$ is called a {\it complex harmonic differential}.

\section{The quasiconformal variations of the cross-ratio}

In this section we estimate the change in the cross-ratio of a quadruple of points under quasiconformal variations. This is the main technical ingredient in the paper.

Given four distinct points $a,b,c,d\in\hat{\mathbb{C}}$, define the {\it cross-ratio}
$$
cr(a,b,c,d)=\frac{(a-c)(b-d)}{(a-d)(b-c)}.
$$

By a hyperbolic metric on a Riemann surface we mean a unique conformal metric of curvature $-1$. 
We first give an estimate of the hyperbolic distance $\rho_{\D^*}$ in the punctured unit disk $\mathbb{D}^{*}=\{ z:0<|z|<1\}$. Recall that the infinitesimal form of the hyperbolic metric is $|dz|/(|z|\log 1/|z|)$ (see \cite[\S 1.7]{Ahlfors}).

\begin{lem}
\label{lem:punctured-disk}
Let $\mathbb{D}^{*}$ be equipped with the complete hyperbolic metric $\rho_{\mathbb{D}^{*}}$ and let $0<\beta <1$ be a point on the positive radius. There exists $r=r(\beta )>0$ such that  
$$
|\beta_1|\leq \beta^{e^{-\rho_{\mathbb{D}^{*}} (\beta ,\beta_1)}}
$$ 
for all $\beta_1\in\mathbb{D}^{*}$ with $\rho_{\D^*}(\beta ,\beta_1)< r$.
\end{lem}

\begin{proof}
The universal covering map $\tau :\mathbb{H}\to \mathbb{D}^{*}$ is given by the formula $\tau (z)=e^{iz}$. The hyperbolic metric $\rho_{\D^*}$ on $\mathbb{D}^{*}$ is the push-forward by $\tau$ of the hyperbolic metric $\rho_{\HH}$ on $\HH$. The covering $\tau :\mathbb{H}\to \mathbb{D}^{*}$ is a local isometry. Indeed, there is $r=r(\beta )>0$ such that the covering map $\tau$ is an isometry from the open hyperbolic disk $\Delta$ with the center $i\log \frac{1}{\beta }$ and hyperbolic radius  $r$ in $\HH$ onto the hyperbolic disk $\tau (\Delta )$ in $\D^*$ with the center $\beta$ and radius $r$. An elementary computation shows that the radius $r=r(\beta )>0$ can be chosen to be   $\sinh^{-1}\frac{\pi}{2\log\frac{1}{\beta}}$.

Set $D=\rho_{\mathbb{D}^{*}} (\beta ,\beta_1)$. Note that $\tau (i\log\frac{1}{\beta} )=\beta$ by the definition of $\tau$ and that $D< r$ because $\beta_1\in\tau (\Delta )$. 
Let $z_1\in{\Delta}$ such that $\tau (z_1)=\beta_1$. Then $\rho_{\HH}(z_1,i\log\frac{1}{\beta})=\rho_{\D^*}(\beta ,\beta_1)=D$. The hyperbolic disk $\Delta_1$ with the center $i\log\frac{1}{\beta}$ and radius $D$  intersects the $y$-axis along its diameter with endpoints  $i(\log\frac{1}{\beta} )e^{-D}$ and  $i(\log\frac{1}{\beta} )e^{D}$. Therefore ${\Delta}_1$ is a Euclidean disk with the center $i(\log\frac{1}{\beta} )\cosh D$ and radius $(\log\frac{1}{\beta} )\sinh D$.

Since $z_1$ lies on the boundary of $\Delta_1$ we have 
$$
|Re (z_1)|\leq (\log\frac{1}{\beta})\sinh D
$$
and
$$
Im(z_1)\geq (\log\frac{1}{\beta} )e^{-D}.
$$
This implies 
$$
|\beta_1|=|e^{iz_1}|\leq e^{-(\log\frac{1}{\beta}) e^{-D}}=(e^{-\log\frac{1}{\beta}})^{e^{-D}}=\beta^{e^{-D}}
$$
which gives the lemma. 
\end{proof}

We need the following estimate on the size of the cross ratio of four points under the action of a quasiconformal map of $\C$. This estimate is the main ingredient in our proofs.

\begin{lem}
\label{lem:liouville-complex}
Let $f$ be a $K$-quasiconformal deformation of $\Gamma$ that fixes $0$, $1$ and $\infty$. Given $\epsilon >0$ there exist $\delta =\delta (K) >0$ and $C=C(K)>0$ such that for any four distinct $a,b,c,d\in\hat{\mathbb{C}}$ with
$$
|cr(a,b,c,d)-1|<\delta
$$
we have
$$
| cr(f(a),f(b),f(c),f(d)) -1|\leq |cr(a,b,c,d)-1|^{1/(K+\epsilon) }.
$$
\end{lem}

\begin{proof}
Let $\mu$ be the Beltrami coefficient of $f$. 
Let $f^{t}:\hat{\mathbb{C}}\to\hat{\mathbb{C}}$ be a family of quasiconformal maps with the Beltrami coefficients $t\mu$ for $\{t\in\mathbb{C}:|t|<1/|\mu\|_{\infty}\}$ that fix $0$, $1$ and $\infty$. By the definition  $f^1=f$. 

Define $g(t)=cr(f^t(a),f^t(b),f^t(c),f^t(d))$ for $|t|<1/\|\mu\|_{\infty}$. Since the points $f^t(a)$, $f^t(b)$, $f^t(c)$ and $f^t(d)$ are pairwise distinct in $\hat{\mathbb{C}}$ and depend holomorphically on $t$, the function $g$ is holomorphic in $t$ and it  does not take values $0$, $1$ and $\infty$. Therefore we obtained a holomorphic function 
$$
g:\mathbb{D}_{1/\|\mu\|_{\infty}}\to\mathbb{C}\setminus\{ 0,1\},
$$
where $\mathbb{D}_{1/\|\mu\|_{\infty}}=\{ t\in\mathbb{C}:|t|<1/\|\mu\|_{\infty}\}$. 

By the Schwarz lemma \cite[Theorem 10.5]{BeardonMinda}, the function $g$ is decreasing for the hyperbolic metrics $\rho_{\mathbb{D}_{1/\|\mu\|_{\infty}}}$ and $\rho_{0,1}$ on $\mathbb{D}_{1/\|\mu\|_{\infty}}$ and  $\mathbb{C}\setminus \{ 0,1\}$, respectively. 
Therefore we have
\begin{equation}
\label{eq:estimate}
\rho_{0,1}(g(0),g(1))\leq \rho_{\mathbb{D}_{1/\|\mu\|_{\infty}}}(0,1)=\log K
\end{equation}
because $K=\frac{1+\|\mu\|_{\infty}}{1-\|\mu\|_{\infty}}$. Note that $g(0)=cr(a,b,c,d)$ and $g(1)=cr(f(a),f(b),f(c),f(d))$.

Since $z\mapsto 1-z$ is a conformal self-map of $\mathbb{C}\setminus\{ 0,1\}$, it is an isometry for the hyperbolic distance. By (\ref{eq:estimate}), we get that
\begin{equation}
\label{eq:upper-hyp}
\rho_{0,1}(1-g(0),1-g(1))\leq\log K.
\end{equation}

Denote by $\lambda_{0,1}$ infinitesimal form of the hyperbolic metric 
$\rho_{0,1}$ on $\mathbb{C}\setminus\{ 0,1\}$. Then we have (see \cite[Theorem 1-12, pages 17-18]{Ahlfors})
$$
\log \lambda_{0,1}(z)=-\log |z|-\log\log \frac{1}{|z|}+O(1)
$$
with $O(1)\to 0$ as $z\to 0$. 
Equivalently, for each $C_1<1$  there exists a Euclidean disk $\mathbf{D}_\delta (0)$ with center $0$ and radius $0<\delta<1$ such that for $z\in\mathbf{D}_\delta (0)$
\begin{equation}
\label{eq:hyp_metric}
\lambda_{0,1}(z)\geq C_1\frac{1}{|z-1|\log\frac{1}{|z-1|}}.
\end{equation}

Let $\gamma\in PSL_2(\mathbb{R})$ be such that $\gamma (a,b,c,d)=(1,c^*,\infty , 0)$ where $1<c^{*}$. Then $cr(a,b,c,d)=cr(1,c^*,\infty , 0)=c^{*}$. By the assumption we have $1<c^{*}<1+\delta $. Let $\delta_{\gamma}^t\in PSL_2(\mathbb{R})$ be such that $\hat{f}^t:=\delta_{\gamma}^t\circ f^t\circ \gamma^{-1}$ fixes $0$, $1$ and $\infty$. 
The Beltrami coefficients of the family $\hat{f}^t=\delta_{\gamma}^t\circ f^t\circ \gamma^{-1}$ for $|t|\leq 1$ have norm bounded by $\|\mu\|_{\infty}<1$ and therefore the family  is equicontinuous. Thus $|\hat{f}^t(c^{*})-1|<1/2$ when $\delta >0$ is small enough for all $|t|\leq 1$. Note that $g(t)=\hat{f}^t(c^{*})$ and
thus when $|g(0)-1|$ is small enough then $|g(1)-1|<1/2$. 
Therefore $g(0)-1$ and $g(1)-1$ are points in the unit punctured disk $\mathbb{D}^*$. The  inequality (\ref{eq:hyp_metric}) implies that the distance between $g(0)-1$ and $g(1)-1$ in the punctured disk $\mathbb{D}^*$ is bounded above $\frac{1}{C_1}$ times their distance in $\C\setminus\{0,1\}$ which by (\ref{eq:upper-hyp}) is less than $\frac{1}{C_1}\log K$.

By Lemma \ref{lem:punctured-disk} and invariance of $\rho_{\mathbb{D}^*}$ under the Euclidean rotation in the origin, there is $\delta >0$ small enough such that 
$$
|g(1)-1|\leq |g(0)-1|^{e^{-\frac{1}{C_1}\log K}}=|g(0)-1|^{1/K^{1/C_1}}.
$$
By choosing $\delta >0$ small enough such that $C_1$ close enough to $1$  we obtain the lemma.
\end{proof}

We give an infinitesimal version of the above estimate.

\begin{lem}
\label{lem:infinitesimal-liouville}
Let $\mu (t)$ for $t\in\mathbf{D}_{r_0}(0)$ with $r_0>0$ be a holomorphic family of Beltrami coefficients in $\C$. Let $f^t:\hat{\mathbb{C}}\to\hat{\mathbb{C}}$  be the quasiconformal map with the Beltrami coefficient $\mu (t)$ that fixes $0$, $1$ and $\infty$. Given $0<r<r_0$, let $K_r=\sup_{\{ t: |t|\leq r\}}\frac{1+\|\mu (t)\|_{\infty}}{1-\|\mu (t)\|_{\infty}}$ be the maximal quasiconformal constant of $f^t$ for $|t|\leq r$. 

Given $\epsilon >0$ and $0<r<r_0$ there exist $\delta_1 =\delta_1 (K_r) >0$ and $C=C(r)>0$ such that for any distinct $a,b,c,d\in\hat{\C}$   with
$$
|cr(a,b,c,d)-1|<\delta_1
$$
we have, for all $|t|\leq r$,
$$
\Big{|}\frac{d}{dt} cr(f^t(a),f^t(b),f^t(c),f^t(d)){|} \Big{|}\leq C|cr(a,b,c,d)-1|^{1/(K_r+\epsilon)} \log\frac{1}{|cr(a,b,c,d)-1|}.
$$
\end{lem}

\begin{proof}
The function $g(t)=cr (f^t(a),f^t(b),f^t(c),f^t(d))$ is a holomorphic map from the disk $\{ |t|<r_0\}$ to $\C\setminus\{ 0,1\}$. Let $\lambda_{\{ |t|<r_0\}}$ and $\lambda_{0,1}$ denote the infinitesimal forms of the corresponding hyperbolic metrics. The Schwarz lemma implies that
$$
\lambda_{0,1}(g(t))|g'(t)|\leq \lambda_{\{ |t|<r_0\}} (t).
$$
Then there exists $C=C(r)>0$ such that, for $|t|\leq r$,
\begin{equation}
\label{eq:schwarz}
|g'(t)|\leq C |g(t)-1|\log\frac{1}{|g(t)-1|}.
\end{equation}
Note that $g'(t)=\frac{d}{dt}cr(f^t(a),f^t(b),f^t(c),f^t(d))$. 
Lemma \ref{lem:liouville-complex} implies that there exists $\delta_1 (K_r) >0$ such that
$$
|cr(f^t(a),f^t(b),f^t(c),f^t(d))-1|\leq |cr(a,b,c,d)-1|^{\frac{1}{K_r+\epsilon}}
$$
for all $a$, $b$, $c$ and $d$ in  $\mathbb{C}$ with $|cr(a,b,c,d)-1|<\delta_1(K)$ 
which together with the inequality (\ref{eq:schwarz}) finishes the proof.
\end{proof}

\section{The Liouville map and uniform H\"older topology}
\label{sec:Holder}

In this this section we define the Liouville map, the space of bounded geodesic currents and the space of bounded H\"older distributions for the Riemann surface $X$.

We identify the conformally hyperbolic Riemann surface $X$ with $\HH /\Gamma$, where $\Gamma$ is a Fuchsian group acting on the upper half-plane $\HH$. The space of oriented geodesics $G(\tilde{X})=(\partial_{\infty}\tilde{X}\times \partial_{\infty}\tilde{X})\setminus\mathrm{diagonal}$ is identified with $(\hat{\R}\times\hat{\R})\setminus\mathrm{diagonal}$. 
Fix a reference point $z_0\in\HH$ and define an {\it angle distance} 
$d(x,y)$ for $x,y\in \hat{\R}=\R\cup\{\infty\}$ with respect to $z_0$ to be the smaller of the two angles between the geodesic rays starting at $z_0$ and ending at $x$ and $y$, respectively. Even though the distance $d$ depends on the choice of the reference point $z_0$, for any two choices the identity map is bi-Lipschitz for the two distances.
The angle distance 
induces the product metric on $G(\tilde{X})$ via the identification  with $(\hat{\R}\times\hat{\R})\setminus\mathrm{diagonal}$ called the {\it angle metric}. Two angle metrics given by two different reference points are bi-Lipschitz.

  A {\it geodesic current} for $X$ is a positive Radon measure on $G(\tilde{X})=(\hat{\R}\times\hat{\R})\setminus\mathrm{diagonal}$ that is invariant under the action of the covering group $\Gamma$ and the change of the orientation of the geodesics (see \cite{Bonahon}). The {\it space of geodesic currents} is denoted by $\mathcal{G}(X)$. The {\it Liouville measure} $L_{\tilde{X}}$ for ${X}$ is the unique (up to scalar multiple) full support Radon measure that is invariant under the isometries of the universal covering $\tilde{X}$ and the change of the orientation of the geodesics. More precisely, the Liouville measure of a Borel set $A\subset (\hat{\R}\times\hat{\R})\setminus\mathrm{diagonal}=G(\tilde{X})$ is given by (see \cite{Bonahon})
 $$
 L_{\tilde{X}}(A)=\iint_{A}\frac{dxdy}{(x-y)^2}.
 $$
 When $A=[a,b]\times [c,d]\subset G(\tilde{X})$ is a {\it box of geodesics} with one endpoint in $[a,b]\subset\hat{\mathbb{R}}$ and another endpoint in $[c,d]\subset\hat{\mathbb{R}}$ then
 $$
 L_{\tilde{X}}([a,b]\times [c,d])=\log cr(a,b,c,d).
 $$ 
 It is evident from the definition that $L_{\tilde{X}}([a,b]\times [c,d])=L_{\tilde{X}}([c,d]\times [a,b])$.
 
 A quasiconformal map $f:X\to X_1$ induces a quasisymmetric map of the ideal boundaries of the universal covers $\tilde{X}$ and $\tilde{X}_1$ of $X$ and $X_1$, respectively. The induced map in turn induces a homeomorphism of the space of geodesics $G(\tilde{X})$ and $G(\tilde{X}_1)$ which is equivariant for the actions of the covering groups. The Teichm\"uller equivalence class  $[f:X\to X_1]$ induces the pull-back of the Liouville measure $L_{X_1}$ to the space of geodesics $G(\tilde{X})$ denoted by $L_{[f]}$ under the induced homeomorphism of $G(\tilde{X})$ and $G(\tilde{X}_1)$ . In general, $L_{[f]}$ is invariant under $\pi_1(X)$ but not invariant under all isometries of $\tilde{X}$ and thus different from $L_X$.
 In fact, $L_{[f]}$ completely recovers the Riemann surface $X_1$ and the Teichm\"uller class of $f:X\to X_1$.

 Bonahon \cite{Bonahon} introduced the {\it Liouville map} 
 $$
 \mathcal{L}:\T (X)\to \mathcal{G}(X)
 $$
 by
 $$
 \mathcal{L}([f])=L_{[f]}.
 $$
Bonahon \cite{Bonahon} used the Liouville map in order to give an alternative description of the Thurston boundary to the Teichm\"uller space of a compact surface. When $X$ is a compact surface the space of geodesic currents $\mathcal{G}(X)$ is equipped with the standard weak* topology for which the Liouville map is an embedding onto its image (see Bonahon \cite{Bonahon}).
 Bonahon and the second author \cite{BonahonSaric} introduced the {\it uniform weak* topology} on the space of geodesic currents in order to introduce a Thurston boundary to Teichm\"uller spaces of arbitrary Riemann surfaces. This is a simplification of the topology that was introduced by the second author \cite{Saric}.
 
Let $\zeta :G(\tilde{X})\to\C$ be a continuous function with a compact support. Define a semi-norm
$$
\|\alpha\|_{\zeta}=\sup_{\gamma\in PSL_2(\R )}\Big{|}\iint_{G(\tilde{X})}\zeta\circ\gamma d\alpha\Big{|}.
$$
 The uniform weak* topology is defined using a family of semi-norms $\{\|\cdot\|_{\zeta}\}$ where $\zeta$ ranges over all continuous functions with compact support (see \cite{BonahonSaric}). 
 
 The space of {\it bounded geodesic currents} $\mathcal{G}_b(X)$ for $X$ consists of all $\alpha\in\mathcal{G}(X)$ such that, for each continuous function $\zeta :G(\tilde{X})\to\mathbb{C}$ with compact support,
 $$
 \|\alpha\|_{\zeta}<\infty .
 $$
 Bonahon and the second author \cite{BonahonSaric} proved that the Liouville map
 $$
 \LLL :\T (X)\to\mathcal{G}_b(X)
 $$
 is an embedding onto its image when $\mathcal{G}_b(X)$ is equipped with the uniform weak* topology and they introduced Thurston boundary to $\T (X)$.
 
 The description of the uniform weak* topology is a significant simplification of the topology introduced by the second author in \cite{Saric}, \cite{Saric1}. The first author \cite{Dong} considered the topology induced by the family of semi-norms $\{\|\cdot\|_{\xi}\}$ where $\xi :G(\tilde{X})\to\C$ is a H\"older continuous function with compact support in order to study the differentiability of the Liouville map. We use this new H\"older topology to give another proof that the Liouville map is real analytic in an appropriate  sense which was proved by Otal \cite{Otal} using a more complicated family of semi-norms from \cite{Saric1} (which involve supremum over all H\"older continuous functions with bounded H\"older norm while we take supremum only over all $\gamma\in PSL_2(\R )$).

Let $H(\tilde{X})$ be the space   of all H\"older continuous functions $\xi: G(\tilde{X})\to\C$ with respect to the product metric on $G(\tilde{X})$ that are of compact support. 
 A linear functional $\W : H(\tilde{X})\to\C$ is said to be {\it bounded} if, for every $\xi\in H(\tilde{X})$,
$$
\| \W \|_{\xi}:=\sup_{\gamma\in PSL_2(\R )} |\W (\xi\circ\gamma )|<\infty .
$$ 
The space of {\it bounded H\"older distributions} $\HHH_{b}(X)$ is the space of all bounded complex linear functionals on the space of H\"older continuous functions $H(\tilde{X})$ with compact support in $G(\tilde{X})$. The topology on $\HHH_{b}(X)$ is induced by the finite intersections of balls for the semi-norms. This makes $\HHH_{b}(X)$ a metrizable topological vector space (see \cite[\S 2.3]{BonahonSaric}). 

Since the space of bounded geodesic currents is a subset of the space of bounded H\"older distributions $\HHH_{b}(X)$, we can consider the Liouville map
$$
\mathcal{L} :\T (X)\to \HHH_{b}(X).
$$
It is not hard to see that the Liouville map is a homeomorphisms onto its image in $\HHH_{b}(X)$ by noting that continuous functions are well approximated by H\"older continuous functions (see \cite[Theorem 4.2.1]{Dong}).

\section{The complexification of the Liouville map}

In this section we define an extension of the Liouville map $\LLL :\T (X)\to\mathcal{H}_b(X)$ to an open neighborhood of $\T (X)$ in the quasi-Fuchsian space $\mathcal{QF}(X)$. The image of the Teichm\"uller space $\T (X)$ under $\mathcal{L}$ consists of geodesic currents, i.e. positive Radon measures on $G(\tilde{X})$. When $[\mu ]\in\mathcal{QF}(X)\setminus\T (X)$ it will turn out that the image $\mathcal{L}([\mu ])$ is not a countably additive measure but rather a linear functional on $H(\tilde{X})$.  Moreover the neighborhood of $\T (X)$ for the extension will depend on a particular function $\xi\in H(\tilde{X})$ for which we define the extension.

Note that any H\"older continuous function with compact support can be written as a finite sum of H\"older continuous functions whose supports are in  boxes of geodesics using  the partition of unity. Therefore we only need to extend $\mathcal{L}([\mu ])(\xi )$ for the case when the support of $\xi$ is in a box of geodesics.

Fix $\lambda$ with $0<\lambda \leq 1$ and $C>0$. Let $\xi :G(\tilde{X})\to\mathbb{C}$ be a $\lambda$-H\"older continuous function. By the definition, the {\it H\"older constant} of $\xi$ is $\sup_{g_1\neq g_2}\frac{|\xi (g_1)-\xi (g_2)|}{d(g_1,g_2)^{\lambda}}$, where the supremum is over all $g_1,g_2\in G(\tilde{X})$ and $d$ is the angle metric. For a box of geodesics $[a,b]\times [c,d]$, denote by 
$H_{\lambda ,C}([a,b]\times [c,d])$ the family of $\lambda$-H\"older continuous functions with the H\"older constant at most $C$ and support in $[a,b]\times [c,d]$. 

Given $n>0$, divide the intervals $[a,b]$ and $[c,d]$ into $2^n$ subintervals that are of equal size with respect to the angle distance. Let $a_i$ and $c_j$ be the division points of $[a,b]$ and $[c,d]$, respectively. In particular $a_0=a$, $a_{2^n}=b$, $c_0=c$ and $c_{2^n}=d$. The box $[a,b]\times [c,d]$ is partitioned into $4^n$ sub-boxes $[a_{i-1},a_i]\times [c_{j-1},c_j]$ for $i,j=1,2,\ldots ,2^n$ whose Liouville measure satisfies
$$
L_{\tilde{X}}([a_{i-1},a_i]\times [c_{j-1},c_j])\leq C_1 4^{-n}
$$
where $C_1>0$ is a constant that depends on $L_{\tilde{X}}([a,b]\times [c,d])$ (by Lemma \ref{lem:boxes} below).

 Note that for $[f^{\mu}]\in \T (X)$, $\xi\in H(\tilde{X})$ and $\gamma\in PSL_2(\mathbb{R})$ we have
$$
\mathcal{L}([f^{\mu}])(\xi\circ\gamma)=\iint_{G(\tilde{X})}\xi dL_{f^{\mu}\circ\gamma^{-1}}
$$
and consequently
\begin{equation}
\label{eq:limit}
\mathcal{L}([f^{\mu}])(\xi\circ\gamma)=\lim_{n\to\infty}\sum_{i,j=1}^{2^n}\xi (a_i,c_j) \log cr(f^{\mu}\circ\gamma^{-1}(a_{i-1}, a_{i},c_{j-1},c_{j})).
\end{equation}
The sum on the right of the above formula is a finite approximation of the integral and note that $\log cr(f^{\mu}\circ\gamma^{-1}(a_{i-1}, a_{i},c_{j-1},c_{j}))>0$.

\vskip .2 cm

In order to get estimates uniform in $\gamma\in PSL_2(\mathbb{R})$, we first need to be able to partition a box of geodesics of a bounded size into smaller boxes of approximately the same size (up to a multiplicative constant that depends only on the size of the initial box). We have

\begin{lem}
\label{lem:boxes}
There exists a continuous function $C=C(x)>0$ for $x>0$ such that
for each box of geodesics $[a,b]\times [c,d]\subset\hat{\mathbb{R}}\times\hat{\mathbb{R}}\setminus\mathrm{diagonal}=G(\tilde{X})$ and for each $n\geq 1$ there exist partitions $[a_{i-1},a_i]$, $i=1,\ldots ,2^n$ and $[c_{j-1},c_j]$, $j=1,\ldots ,2^n$ of intervals $[a,b]$ and $[c,d]$ satisfying
$$
L_{\tilde{X}}([a_{i-1},a_i]\times [c_{j-1},c_j])\leq C 4^{-n}
$$
with $C=C(L_{\tilde{X}}([a,b]\times [c,d]))$.

Moreover, each subsequent partition is obtained by adding one partition point in each interval of the prior partition.
\end{lem}

\begin{proof}
The Liouville measure is invariant under the action of the group $PSL_2(\mathbb{R})$. We map $[a,b]\times [c,d]$ by an element of $PSL_2(\mathbb{R})$ onto $[-1,0]\times [c^{*},1 ]$, where $c^{*}=1/(2e^{L_{\tilde{X}}([a,b]\times [c,d])}-1)$. 

Let $a_0=-1$, $a_i=-1+i2^{-n}$ for $i=1,\ldots ,2^n$ and $c_0=c^{*}$, $c_j=c^{*}+(1-c^{*})j2^{-n}$ for $j=1,\ldots ,2^n$.  We estimate
$$
\log cr(a_{i-1},a_{i},c_{j-1},c_{j})=\log \frac{(c_{j-1}-a_{i-1})(c_j-a_i)}{(c_j-a_{i-1})(c_{j-1}-a_i)}.
$$
Using the inequality $\log (1+x) \leq x$ for $x>-1$ and the fact that
\begin{align*}
	cr(a_{i-1},a_{i},c_{j-1},c_{j})
	&=\frac{(c_{j-1}-a_{i-1})(c_j-a_i)}{(c_j-a_{i-1})(c_{j-1}-a_i)}\\
	&=1+\frac{(a_i-a_{i-1})(c_j-c_{j-1})}{(c_j-a_{i-1})(c_{j-1}-a_i)},
\end{align*} it is enough to estimate 
$$
\frac{(a_i-a_{i-1})(c_j-c_{j-1})}{(c_j-a_{i-1})(c_{j-1}-a_i)}.
$$
Since $c_{j}-a_{i-1} \geq c^{*}$, $c_{j-1}-a_{i} \geq c^{*}$, $a_i-a_{i-1}= 2^{-n}$ and $c_j-c_{j-1}=(1-c^{*})2^{-n}$, the desired estimate follows with $C=\frac{1-c^*}{(c^*)^2}$. 
\end{proof}

Our goal is to show that the  limit in (\ref{eq:limit}) exists when we replace $[f^{\mu}]\in\T (X)$ with $[f^{\nu}]\in\mathcal{QF}(X)$ that are close to $\T (X)$. 
Denote by $\log z$ the branch of the logarithm defined in $\C\setminus\{ \mathrm{Re}(z)\leq 0\}$ whose imaginary part is in the interval $(-\pi ,\pi )$. 
Given $\epsilon >0$, denote by $\mathcal{V}_{\epsilon}\subset \mathcal{QF}(X)$ the set of all $[f^{\nu}]$ with $\|\nu\|_{\infty}<\epsilon$.

\begin{lem}
\label{lem:conv-dist}
For $0<\lambda \leq 1$ and $C>0$ fixed, there exists $\epsilon >0$ such that the limit
$$
\lim_{n\to\infty}\sum_{i,j=1}^{2^n}\xi (a_i,c_j) \log cr(f^{\nu}\circ\gamma^{-1}(a_{i-1}, a_{i},c_{j-1},c_{j}))
$$
exists for all $[f^\nu ]\in\mathcal{V}_{\epsilon}$, $\xi\in H_{\lambda ,C}([a,b]\times [c,d])$ and $\gamma\in PSL_2(\mathbb{R})$. The convergence is uniform and the limit is bounded in $\mathcal{V}_{\epsilon}\times PSL_2(\mathbb{R})$.
\end{lem}

\begin{proof} Let $I_n:=\sum_{i,j=1}^{2^n}\xi (a_i,c_j) \log cr(f^{\nu}\circ\gamma^{-1}(a_{i-1}, a_{i},c_{j-1},c_{j}))$. The partition for $I_{n+1}$ is obtained by dividing each box $[a_{i-1},a_i]\times [c_{j-1},c_j]$ into four sub-boxes $[a_{(i-1)k}, a_{ik}]\times [c_{(j-1)k},c_{jk}]$ for $k=1,\ldots ,4$ of Liouville measures at most $C_14^{-n-1}$ for some universal constant $C_1$ by Lemma \ref{lem:boxes}. 

Lemma \ref{lem:liouville-complex} implies that there is $\epsilon >0$ such that the imaginary part of the cross-ratio $cr(f^{\nu}\circ\gamma^{-1}(a_{i(k-1)}, a_{ik},c_{j(k-1)},c_{jk}))$ is small when $n$ is large and $[f^{\nu}]\in\mathcal{V}_{\epsilon}$. Thus the logarithm of the cross-ratio of $f^{\nu}\circ\gamma^{-1}(a_{i-1}, a_{i},c_{j-1},c_{j})$ is a finitely additive complex measure on these boxes. More precisely, we have
$$
\log cr(f^{\nu}\circ\gamma^{-1}(a_{i-1}, a_{i},c_{j-1},c_{j}))=\sum_{k=1}^4\log cr(f^{\nu}\circ\gamma^{-1}(a_{(i-1)k}, a_{ik},c_{(j-1)k},c_{jk})).
$$
and
$$
|I_{n+1}-I_n|\leq\sum_{i,j=1}^{2^n}\sum_{k=1}^4|\xi (a_{ik},c_{jk})-\xi (a_i,c_j)|\cdot |\log cr(f^{\nu}\circ\gamma^{-1}(a_{(i-1)k}, a_{ik},c_{(j-1)k},c_{jk}))|.
$$

By Lemma \ref{lem:liouville-complex}, we have
$$|\log cr(f^{\nu}\circ\gamma^{-1}(a_{(i-1)k}, a_{ik},c_{(j-1)k},c_{jk}))|\leq C_24^{-\omega (n+1)}$$ 
where $\omega \to 1$ as $\epsilon \to 0$ and $C_2$ is some constant which depends on $\epsilon$ and $C_1$.
Since $\xi\in H_{\lambda ,C}([a,b]\times [c,d])$ we have
$$
|\xi (a_{ik},c_{jk})-\xi (a_i,c_j)|\leq C2^{-\lambda n}.
$$
This implies $|I_{n+1}-I_n|\leq C_3 4^{(1-\frac{\lambda}{2}-\omega )n}$ for some constant $C_3$ and
$$
I_1+\sum_{n=1}^{\infty} (I_{n+1}-I_n)
$$
converges uniformly for small $\epsilon >0$ since $1-\frac{\lambda}{2}-\omega <0$. The partial sum of the first $n$ terms of the series is $I_{n+1}$ and the lemma is proved.
\end{proof}

Let $[f^{\mu}]\in \T (X)$ and $\xi\in H(\tilde{X})$. It would be ideal to find a neighborhood $\mathcal{V}$ in $\mathcal{QF}(X)$ of the point $[f^{\mu}]$ such that the Liouville map $\LLL$ extends to $\mathcal{V}$. Otal \cite{Otal} noted that this seems impossible and the proof of Lemma \ref{lem:conv-dist} uses the fact that $1-\omega-\frac{\lambda}{2}<0$ for $\omega$ close to $1$ which is dependent on the size of $\lambda >0$ and in turn on the size of $\epsilon$. For our purposes,
it is enough to find a neighborhood $\mathcal{V}_{\epsilon (\lambda )}([f^\mu])$ of $[f^{\mu}]$ that depends on $\lambda$ and define the extension on $\lambda$-H\"older continuous functions $\xi$ with compact support in $ G(\tilde{X})=G(\HH )$.

Define $\mathcal{V}_{\epsilon (\lambda)}([\mu ])$ to consist of all $[f^{\mu +\nu}]\in\mathcal{QF}(X)$ with $\|\nu\|_{\infty}<\epsilon$, where $\epsilon =\epsilon (\lambda )>0$ is to be determined later. Define $f^{\eta}=f^{\mu +\nu}\circ (f^{\mu})^{-1}$ and note that $\|\eta\|_{\infty}\leq \frac{\|\nu\|_{\infty}}{1-|\mu \|_{\infty}}$. In particular, $\|\nu\|_{\infty}\to 0$ if and only if $\|\eta\|_{\infty}\to 0$. 

Let $\HHH^\lambda_b (X)$ be the space of complex linear functionals $\mathbf{W}$ on the space $H_{\lambda}(\tilde{X})$ of $\lambda$-H\"older continuous function with  compact support in $G(\tilde{X})$ that satisfy
$$
\|\mathbf{W}\|_{\xi}=\sup_{\gamma\in PSL_2(\R )}|\mathbf{W}(\xi\circ\gamma )|<\infty
$$
for all $\xi\in H_{\lambda}(\tilde{X})$. 

\begin{thm}
\label{thm:extension}
Given $[f^{\mu}]\in\mathcal{T}(X)$ and $0<\lambda\leq 1$, there exists $\epsilon =\epsilon (\lambda ) >0$ such that the Liouville map 
$\mathcal{L}: \T (X)\to\mathcal{H}_b(X)$ extends to a continuous map 
$$
\hat{\mathcal{L}}:\mathcal{V}_\epsilon ([\mu])\to\mathcal{H}^\lambda_b (X).
$$
\end{thm}

\begin{proof}
We fix isometric identifications of $\tilde{X}$ and  $\tilde{X}_1:=f^{\mu}(\tilde{X})$ with the upper half-plane $\HH$. Fix $[f^{\mu +\nu}]\in\mathcal{V}_{\epsilon}([\mu ])$.
Let $\xi\in H_{\lambda}(\tilde{X})$ and without loss of generality  we can assume that the support of $\xi$ is in a box of geodesics $[a,b]\times [c,d]$. Let $\gamma\in PSL_2(\mathbb{R})$ be arbitrary and $\alpha_{\gamma}\in PSL_2(\mathbb{R})$ such that $\alpha_{\gamma}\circ f^{\mu +\nu}\circ\gamma^{-1}$ fixes $0$, $1$ and $\infty$. 
Then
$\alpha_{\gamma}\circ f^{\mu +\nu}\circ\gamma^{-1}=f^{\eta}\circ f^{\mu_1}$ with $f^{\mu_1}=\delta_{\gamma}\circ f^{\mu}\circ\gamma^{-1}:\mathbb{H}\to\mathbb{H}$ where $\delta_{\gamma}\in PSL_2(\R )$ is chosen to normalize $f^{\mu_1}$ to fix $0$, $1$ and $\infty$. Necessarily $f^{\eta}$ fixes $0$, $1$ and $\infty$. The corresponding Beltrami coefficients satisfy 
$\|\mu_1\|_{\infty}=\|\mu\|_{\infty}$ and $\|\eta\|_{\infty} \leq \frac{\|\nu\|_{\infty}}{1-\|\mu\|_{\infty}}$. 

By varying $\gamma\in PSL_2(\mathbb{R})$, we obtain a family of maps $\{ f^{\mu_1}\}_{\gamma}$.
Since each $f^{\mu_1}$ fixes $0$, $1$ and $\infty$, and the Beltrami coefficients have constant norms, the whole family and its inverses are H\"older homeomorphisms with  H\"older exponent $\lambda_1>0$ and H\"older constant $C_1<\infty$ in the spherical metric, where a single $\lambda_1>0$ and $C_1<\infty$ work for the whole family. 
It follows that $\xi\circ (f^{\mu_1})^{-1}\in H_{\lambda_2}(\tilde{X}_1)$ where $\lambda_2>0$ holds for the whole family and the support of each $\xi\circ (f^{\mu_1})^{-1}$ is in the box $[a_*,b_*]\times [c_*,d_*]:=f^{\mu_1}([a,b]\times [c,d])$ which depends on $\gamma\in PSL_2(\R )$.

The construction starts with an approximation of $\xi\circ (f^{\mu_1})^{-1}$ by step functions $\xi_n$ (see \cite{Saric1}, \cite{Dong}). Partition the box $[a_*,b_*]\times [c_*,d_*]$ into $4^n$ sub-boxes using Lemma \ref{lem:boxes} as in the proof of Lemma \ref{lem:conv-dist}. The sub-boxes $[a_{i-1},a_i]\times [c_{j-1},c_j]$ of $[a_*,b_*]\times [c_*,d_*]$ intersect in at most their boundaries which have Liouville measure $0$. The step function $\xi_n$ is defined to have a constant value $\xi\circ (f^{\mu_1})^{-1} (a_i,c_j)$ on the box   $[a_{i-1},a_i]\times [c_{j-1},c_j]$ for $i,j=1,2,\ldots ,2^n$ and zero elsewhere. The ambiguity for the definition of $\xi_n$ on the intersections of  the boxes is not important because the Liouville map is defined using the integration. 

For $\xi\in H_{\lambda}([a,b]\times [c,d])$ we have
\begin{equation}
\iint_{G(\tilde{X})} \xi \circ \gamma dL_{f^{\mu }}=\iint_{G(\tilde{X})} \xi\circ (f^{\mu_1})^{-1} dL_{\tilde{X}_1}=\lim_{n\to\infty} \iint_{G(\tilde{X})} \xi_n  dL_{\tilde{X}_1}
\end{equation}
and
\begin{equation}
\label{eq:step-int}
\iint_{G(\tilde{X})} \xi_n  dL_{\tilde{X}_1}=\sum_{i,j=1}^{2^n}\xi\circ (f^{\mu_1})^{-1}  (a_i,c_j) \log cr(a_{i-1}, a_{i},c_{j-1},c_{j}).
\end{equation}

The expression in (\ref{eq:step-int}) and $f^{\mu +\nu}\circ\gamma=f^{\eta}\circ f^{\mu_1}$ motivates the definition
\begin{equation}
\label{eq:step_complex}
\hat{\LLL}[f^{\mu +\nu}](\xi_n\circ\gamma ) =\sum_{i,j=1}^{2^n}\xi_n (a_i,c_j) \log cr f^{\eta}(a_{i-1}, a_{i},c_{j-1},c_{j}).
\end{equation}

By Lemma \ref{lem:conv-dist}, there is $\epsilon >0$ such that $\hat{\LLL}[f^{\mu +\nu}](\xi_n\circ\gamma )$ converges uniformly on $\mathcal{V}_{\epsilon}([\mu ])\times PSL_2(\mathbb{R})$ to a bounded function which is the desired extension. In Theorem \ref{thm:main1} in the next section we prove that $\hat{\LLL}$ is holomorphic on each affine disk in $\mathcal{V}_{\epsilon}([\mu ])$. By \cite[Theorem 14.9, page 198]{Chae}, it follows that $\hat{\LLL}$ is continuous.
\end{proof}

\section{The Liouville embedding is real analytic}

In this section we prove that the Liouville embedding is real analytic. To do so, we show that the complexification $\hat{\LLL}:\mathcal{V}_{\epsilon}([\mu ])\to \HHH^{\lambda}_b(X)$ is holomorphic.

\begin{thm}
\label{thm:main1}
Let $0<\lambda \leq 1$ and $[\mu ]\in\T (X)$.
For $\epsilon =\epsilon (\lambda )>0$, the complexified Liouville embedding 
$$
\hat{\LLL} :\mathcal{V}_{\epsilon}([\mu ])\to \HHH_b^{\lambda}(X)
$$
is holomorphic.
\end{thm}

\begin{proof}
It is enough to prove that $\hat{\LLL}$ is holomorphic on each affine disk $D=\{ \phi^0+t\phi:|t|<r \}$ in a chart of $\mathcal{QF}(X)$ corresponding to $\mathcal{V}_{\epsilon}([\mu ])$ and that $\hat{\LLL}$ is bounded (see \cite[Theorem 14.9, page 198]{Chae}). By the proof of Theorem \ref{thm:extension}, $\hat{\LLL}(\mathcal{V}_{\epsilon}([\mu ]))$ is bounded and it remains to prove that $\hat{\LLL}$ is holomorphic on each affine disk.

If $\eta^0$ and $\eta$ are the complex harmonic Beltrami coefficients corresponding to $\phi^0$ and $\phi$ for the Riemann surface $f^\mu(X)$, then the complex harmonic Beltrami coefficient $\eta_t$ corresponding to $\phi^0+t\phi$ satisfies
$$
\eta_t=\eta^0+t\eta .
$$

The inverse of the affine disk $D$ in $\mathcal{QF}(X)$ is represented by the quasiconformal maps
$
f^{\eta_t}\circ f^\mu.$ Let $\nu_t$ be the Beltrami coefficient of $
f^{\eta_t}\circ f^\mu\circ\gamma^{-1}.$

There exists $\delta_{\gamma}\in PSL_2(\R )$ such that 
$$
f^{\eta_t}\circ f^{\mu}\circ\gamma^{-1}=f^{\hat{\eta}_t}\circ f^{\mu_1}
$$
where $f^{\hat{\eta}_t}$ and $f^{\mu_1}=\delta_{\gamma}\circ f^{\mu}\circ\gamma^{-1}$ fix $0$, $1$ and $\infty$.

We need to prove that there exists a complex linear map $D_{[\nu_t ]}\hat{\LLL}:T_{[\nu_t ]}\mathcal{QF} (X)\to\HHH^{\alpha}_b(X)$ such that
\begin{equation}
\label{eq:derivative-uniform}
\lim_{h\to 0}\frac{\|\hat{\LLL}([\nu_{t+h}])-\hat{\LLL}([\nu_t ])-hD_{[\nu_t ]}\hat{\LLL}(\eta )\|_{\xi}}{|h|}=0
\end{equation}
for all $\xi\in H_{\alpha}(\tilde{X})$ with the limit being uniform in $\gamma\in PSL_2(\mathbb{R})$. 

By the use of the partition of unity, we can assume that the support of $\xi$ is in a box of geodesics $[a,b]\times [c,d]$. As before, the support of the family $\xi\circ (f^{\mu_1})^{-1}$ is in a box of geodesics $[a_*,b_*]\times [c_*,d_*]$. We partition $[a_*,b_*]\times [c_*,d_*]$ into $4^n$ sub-boxes and define a step function $\xi_n$ to have value $\xi\circ (f^{\mu_1})^{-1}(a_i,c_j)$ on $[a_{i-1},a_i]\times [c_{j-1},c_j]$ and zero elsewhere. Let $I_n (t)=\iint_{G(\tilde{X})}\xi_n dL_{f^{\hat{\eta}_t}}=\sum_{i,j=1}^{2^n}\xi\circ (f^{\mu_1})^{-1} (a_i,b_j) \log cr(f^{\hat{\eta}_t}(a_{i-1}, a_{i},c_{j-1},c_{j}))$. We proved in Theorem \ref{thm:extension} that $I_1+\sum_{n=1}^{\infty}(I_{n+1}-I_n)=\hat{\LLL}([\nu_t ])(\xi\circ \gamma )$. 

Denote by $\dot{I}_n:=\frac{d}{dt}I_n(t)$ and define
$$
D_{[\nu_t ]}\hat{\LLL}(\eta )(\xi\circ\gamma)=\dot{I}_1+\sum_{n=1}^{\infty}(\dot{I}_{n+1}-\dot{I}_n).
$$
The argument in Lemma \ref{lem:conv-dist} estimates the series $\sum_{n=1}^{\infty}(I_{n+1}-I_n)$ by $\sum_{n=1}^{\infty}4^{n(\omega +\frac{\lambda_1}{2}-1)}<\infty$. An analogous approach using the estimate from Lemma \ref{lem:infinitesimal-liouville} gives an upper bound $\sum_{n=1}^{\infty}n4^{n(\omega +\frac{\lambda_1}{2}-1)}<\infty$ to $\sum_{n=1}^{\infty}|\dot{I}_{n+1}-\dot{I}_n|$. This shows that $D_{[\nu_t ]}\hat{\LLL}(\eta )(\xi\circ\gamma)$ is well-defined with an upper bound independent of $\gamma$. Thus $D_{[\nu_t]}\hat{\LLL}(\eta )=\frac{d}{dt}\hat{\LLL}([\nu_t ])\in\HHH_b^{\lambda}(X)$ and it is complex linear as a limit of complex linear $\dot{I}_n$. 

It remains to prove (\ref{eq:derivative-uniform}). To simplify the notation we write $I(t):=\hat{\LLL}([\nu_t ])(\xi\circ\gamma )$ and $\dot{I}(t)=\frac{d}{dt}I(t)=D_{[\nu_t ]}\hat{\LLL}(\eta )(\xi\circ\gamma )$. Then  (\ref{eq:derivative-uniform}) becomes
$$
\lim_{h\to 0}\sup_{\gamma\in PSL_2(\R )} \frac{|I(t+h)-I(t)-h\dot{I}(t)|}{|h|}=0.
$$
It is enough to prove that for any $\epsilon >0$ there is $n>0$ such that
\begin{equation}
\label{eq:uniform-n}
\lim_{h\to 0}\sup_{\gamma\in PSL_2(\R )} \frac{|I_n(t+h)-I_n(t)-h\dot{I}_n(t)|}{|h|}<\frac{\epsilon}{2}
\end{equation} 
and
\begin{equation}
\label{eq:uniform-limit-n}
\lim_{h\to 0}\sup_{\gamma\in PSL_2(\R )} \frac{|(I_n-I)(t+h)-(I_n-I)(t)-h(\dot{I}_n-\dot{I})(t)|}{|h|}<\frac{\epsilon}{2}
\end{equation}

Inequality (\ref{eq:uniform-n}) holds for any $n$ because the limit is zero by the uniform (in $\gamma\in PSL_2(\R )$) differentiability of $cr (f^{[\nu_t ]}\circ\gamma )(a_{i-1},a_i,c_{j-1},c_j)$ in the variable $t$. To estimate the limit in (\ref{eq:uniform-limit-n}), we divide it into two parts and prove that
\begin{equation}
\label{eq:part1}
\lim_{h\to 0}\sup_{\gamma\in PSL_2(\R )} \frac{|(I_n-I)(t+h)-(I_n-I)(t)|}{|h|}<\frac{\epsilon}{4}
\end{equation}
and
\begin{equation}
\label{eq:part2}
\lim_{h\to 0}\sup_{\gamma\in PSL_2(\R )} \frac{|h(\dot{I}_n-\dot{I})(t)|}{|h|}<\frac{\epsilon}{4}.
\end{equation}

The inequality (\ref{eq:part2}) holds for $n$ large enough because $(\dot{I}-\dot{I}_n)(t)=\sum_{k=n}^\infty (\dot{I}_{k+1}-\dot{I}_k)(t)$ is a tail of a uniformly convergent series. To prove (\ref{eq:part1}) note that by  Mean Value Theorem the limit is bounded above by $\max_s\sum_{k=n}^\infty |(\dot{I}_{k+1}-\dot{I}_k)(s)|$ for $s$  between $t$ and $t+h$. The series is arbitrary small uniformly in $\gamma\in PSL_2(\R )$ for $n$ large because it is the tail of a uniformly convergent series which was established in the first part of the proof  and the theorem is established.
\end{proof}

Since holomorphic maps are equal to the sums of their Taylor series, a direct consequence of the above theorem is

\begin{thm}
\label{thm:main2}
The Liouville embedding 
$$
\LLL :\T (X)\to \HHH_b(X)
$$
is real analytic.
\end{thm}

\end{document}